\documentclass[twoside,11pt]{article}

\usepackage{graphicx}
\usepackage{caption}
\usepackage{subcaption}
\usepackage[top=1in,bottom=1in,left=1in,right=1in]{geometry}
\usepackage[sort&compress]{natbib} 
\usepackage{amsfonts, amsmath, amssymb, amsthm, constants, bbm}
\usepackage{MnSymbol}
\usepackage[linesnumbered,lined,boxed,commentsnumbered]{algorithm2e}
\usepackage{graphicx}

\usepackage{caption}
\SetAlCapSkip{1em}
\usepackage{booktabs}

\usepackage{enumitem}

\usepackage{color}
\definecolor{darkred}{RGB}{100,0,0}
\definecolor{darkgreen}{RGB}{0,100,0}
\definecolor{darkblue}{RGB}{0,0,150}

\usepackage{hyperref}
\hypersetup{colorlinks=true, linkcolor=darkred, citecolor=darkgreen, urlcolor=darkblue}
\usepackage{url}

\newtheorem{prp}{Proposition}

\newtheorem{cor}{Corollary}

{\theoremstyle{remark}

\newtheorem{rem}{Remark}
\newtheorem{asp}{Assumption}
}

\def\beq{\begin{equation}} % \setcounter{equation}{1}}
\def\eeq{\end{equation}}
\def\beqn{\begin{eqnarray*}}
\def\eeqn{\end{eqnarray*}}
\def\Bitem{\begin{itemize}\setlength{\itemsep}{.2in}}
\def\bitem{\begin{itemize}\setlength{\itemsep}{.05in}}
\def\eitem{\end{itemize}}
\def\Benum{\begin{enumerate}\setlength{\itemsep}{.2in}}
\def\benum{\begin{enumerate}\setlength{\itemsep}{.05in}}
\def\eenum{\end{enumerate}}
\def\bmult{\begin{multline*}}
\def\emult{\end{multline*}}
\def\bcenter{\begin{center}}
\def\ecenter{\end{center}}
\def\bframe{\begin{frame}}
\def\eframe{\end{frame}}

\newcommand{\prpref}[1]{Proposition~\ref{prp:#1}}

\newcommand{\secref}[1]{Section~\ref{sec:#1}}
\newcommand{\figref}[1]{Figure~\ref{fig:#1}}
\newcommand{\tabref}[1]{Table~\ref{tab:#1}}
\newcommand{\algref}[1]{Algorithm~\ref{alg:#1}}
\newcommand{\remref}[1]{Remark~\ref{rem:#1}}

\newcommand{\aspref}[1]{Assumption~\ref{asp:#1}}

\def\cA{\mathcal{A}}

\def\cN{\mathcal{N}}

\def\bW{\mathbf{W}}

\def\bbP{\mathbb{P}}

\def\bbR{\mathbb{R}}

\newcommand{\E}{\operatorname{\mathbb{E}}}

\newcommand{\Var}{\operatorname{Var}}

\def\1{\mathbbm{1}}

\definecolor{purple}{rgb}{0.4,.1,.9}

\newcommand\blfootnote[1]{%
  \begingroup
  \renewcommand\thefootnote{}\footnote{#1}%
  \addtocounter{footnote}{-1}%
  \endgroup
}

\begin{document}
\thispagestyle{empty}

\title{RANSAC Algorithms for \\ Subspace Recovery and Subspace Clustering}
\author{Ery Arias-Castro \and Jue Wang}
\date{\normalsize University of California, San Diego}
\maketitle

\blfootnote{The present project was initiated in the context of an {\em Independent Study for Undergraduates} (Math 199).  We acknowledge support from the US National Science Foundation (DMS 1513465).}

\begin{abstract}
We consider the RANSAC algorithm in the context of subspace recovery and subspace clustering.  We derive some theory and perform some numerical experiments.  We also draw some correspondences with the methods of \cite{hardt2013algorithms} and \cite{chen2009spectral}.
\end{abstract}

%\tableofcontents

\section{Introduction} \label{sec:intro}

The Random Sample Consensus (with acronym RANSAC) algorithm of \cite{fischler1981random}, and its many variants and adaptations, are well-known in computer vision for their robustness in the presence of gross errors (outliers).  In this paper we focus on the closely related problems of subspace recovery and subspace clustering in the presence of outliers, where RANSAC-type methods are believed to be optimal, yet too costly in terms of computations when the fraction of inliers is small.  Although this is a well-understood limitation of the RANSAC, we nevertheless establish this rigorously in the present context.  In particular, we derive the performance and computational complexity of RANSAC for these two problems, and perform some numerical experiments corroborating our theory and comparing the RANSAC with other methods proposed in the literature.

\subsection{The problem of subspace recovery} \label{sec:recovery-intro}
Consider a setting where the data consist of $n$ points in dimension $p$, denoted $x_1, \dots, x_n \in \bbR^p$.  It is assumed that $m$ of these points lie on a $d$-dimensional linear subspace $L$ and that the points are otherwise in general position, which means that the following assumption is in place:
\begin{asp}\label{asp:general}
A $q$-tuple of data points (with $q \le p$) is linearly independent unless it includes at least $d+1$ points from $L$.
\end{asp}
(We say that points are linearly dependent if they are so when seen as vectors.)

The points on $L$ are called inliers and all the other points are called outliers.  This is the setting of subspace recovery without noise.  When there is noise, the points are not exactly on the underlying subspace but rather in its vicinity.  In any case, the goal is to recover $L$, or said differently, distinguish the inliers from the outliers.
See \figref{recovery} for an illustration in a setting where the subspace is of dimension $d=2$ in ambient dimension $p=3$.
The goal is to recover the subspace $L$ and/or identify the inlier points.

%\begin{figure}[!ht]
%\centering 
%{\includegraphics[scale=0.65]{recover3}}
%\caption{Subspace recovery problem.  
%In red is the subspace (denoted $L$ in the text).  Inliers (points on the subspace) are represented as red circles.  Outliers (points not on the subspace) are represented as black crosses.}
%\label{fig:recovery}
%\end{figure}

This problem is intimately related to the problem of robust covariance estimation, which dates back decades \citep{maronna1976robust, huber2009robust, tyler1987distribution}, but has attracted some recent attention.  We refer the reader to the introduction of \citep{zhang2014novel} for a comprehensive review of the literature, old and new.
Subspace recovery in the presence of outliers, as we consider the problem here, is sometimes referred to a robust principal components analysis, although there are other meanings in the literature more closely related to matrix factorization with a low-rank component \citep{wright2009robust, candes2011robust}.

\begin{figure}[h!]
\centering
\begin{subfigure}[t]{.4\textwidth}
\centering
\includegraphics[scale=.3]{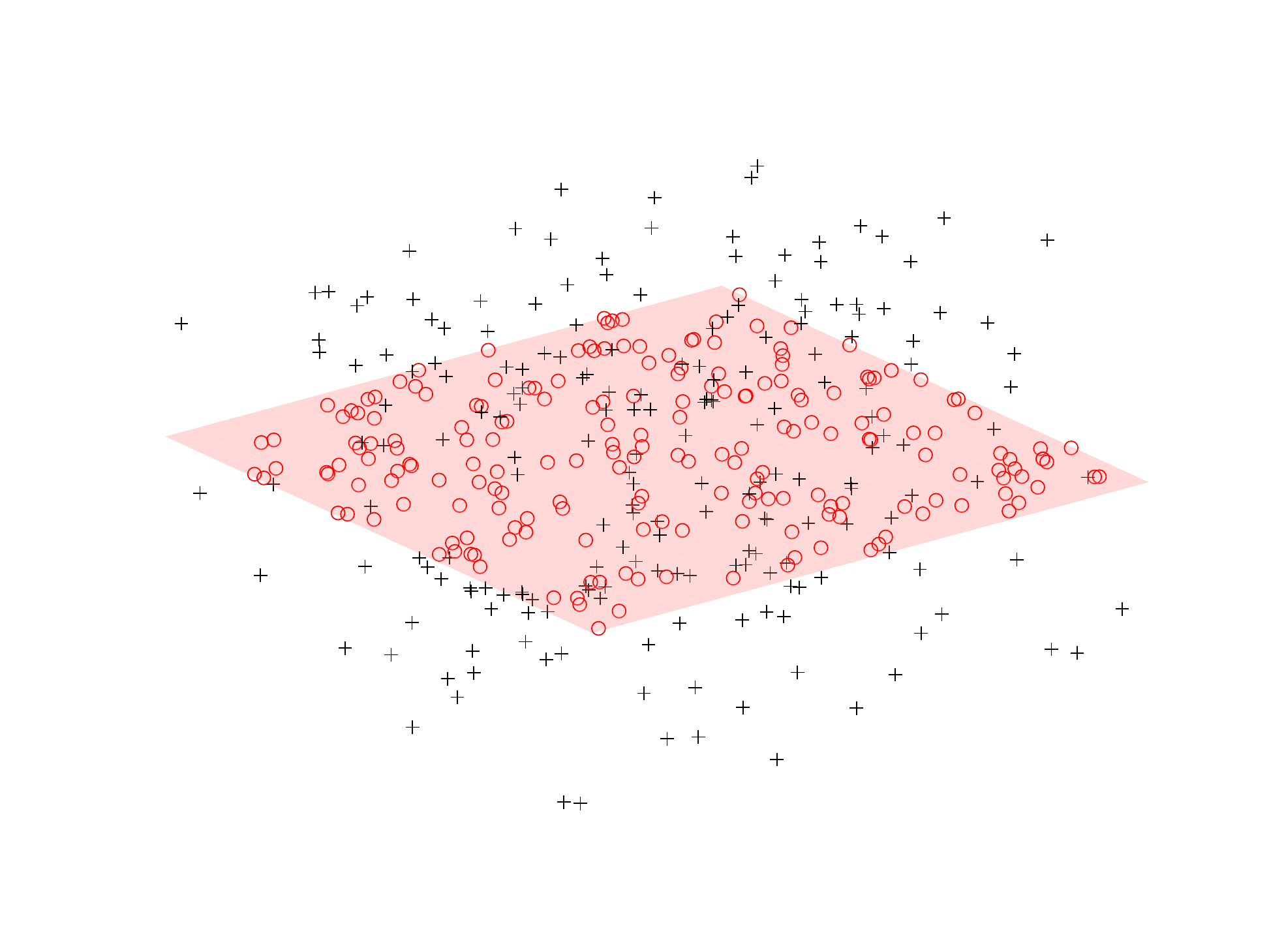}
\caption{Subspace recovery problem.}  
\label{fig:recovery}
\end{subfigure}%
\qquad
\begin{subfigure}[t]{.4\textwidth}
\centering
\includegraphics[scale=0.3]{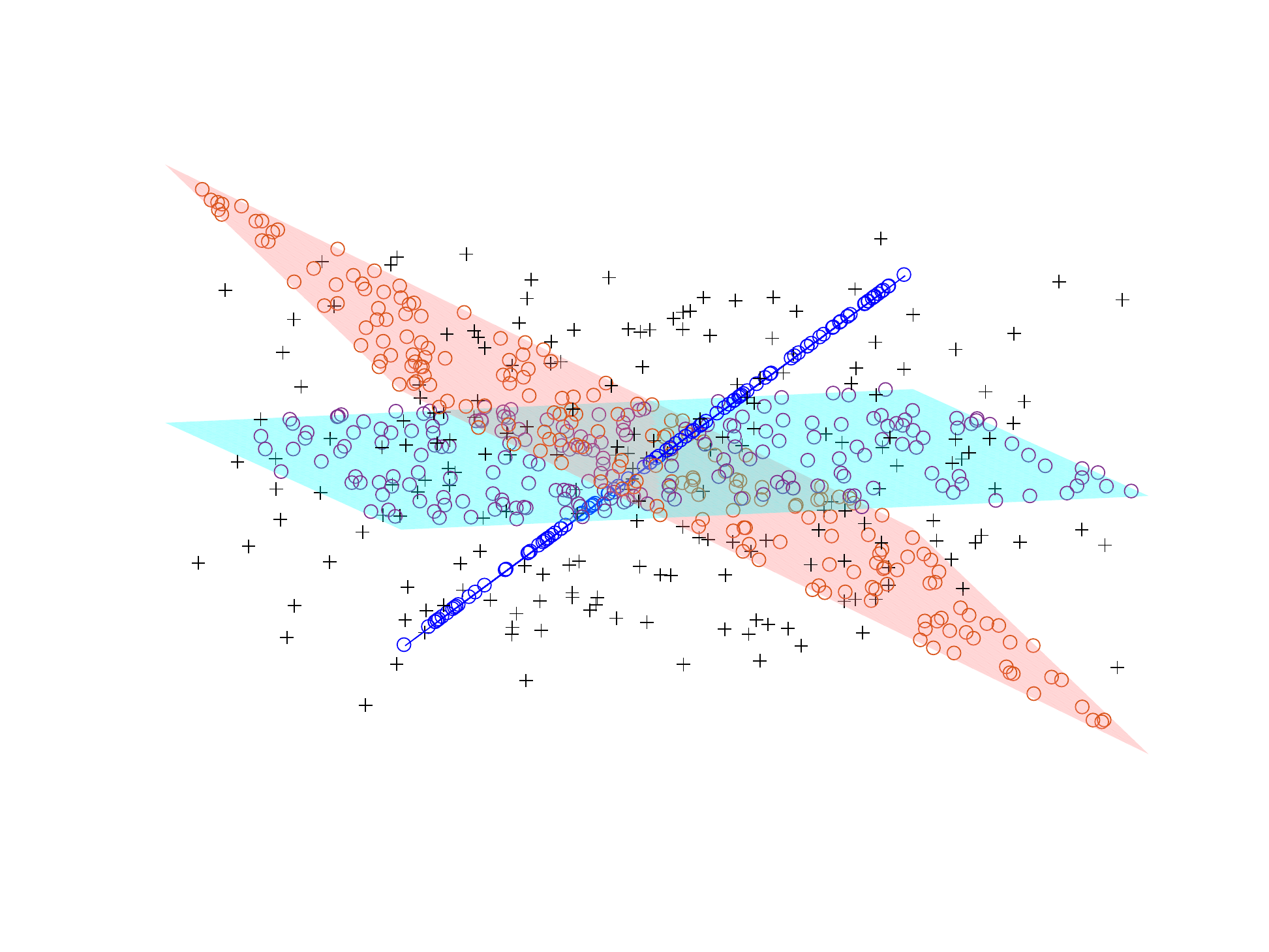}
\caption{Subspace clustering problem.}  
\label{fig:clustering}
\end{subfigure}%
\caption{An illustration of the two settings considered in the paper.}
\end{figure}

\subsection{The problem of subspace clustering} \label{sec:clustering-intro}
Consider a setting where the data consist of $n$ points in dimension $p$, denoted $x_1, \dots, x_n \in \bbR^p$. It is assumed that $m_k$ of these points lie on a $d_k$-dimensional linear subspace $L_k$, where $k = 1, \dots, K$ (so that there are $K$ subspaces in total).  The remaining points are in general position:

\begin{asp}\label{asp:general-clustering}
A $q$-tuple of data points is linearly independent unless it includes at least $d_k+1$ points from $L_k$ for some $k \in \{1, \dots, K\}$.
\end{asp}

In this setting all the points on one of the subspaces are inliers, and all the other points are outliers. 
This is the setting of subspace clustering without noise. 
When there is noise, the inliers are not exactly on the subspaces but in their vicinity.  
See \figref{clustering} for an illustration in a setting where there is one subspace of dimension $d_1 =1$ and two subspaces of dimension $d_2 = d_3 = 2$, in ambient dimension $p=3$.
The goal is to cluster $m_k$ points to their corresponding $L_k$ for all $k \in \{1, \dots, K\}$.

%\begin{figure}[!ht]
%\centering 
%{\includegraphics[scale=0.75]{cluster3}}
%\caption{Subspace clustering problem.  
%In blue is a subspace of dimension 1 (a line) and in cyan and red are two subspaces of dimension 2 (planes).  Inliers (points on one of the subspaces) are represented as circles.  Outliers (points not on one of the subspaces) are represented as crosses.}
%\label{fig:clustering}
%\end{figure}

The problem of subspace clustering has applications in computer vision, in particular, movement segmentation \citep{vidal2011subspace, vidal2005generalized}.

\subsection{Contents}
In \secref{recovery}, we consider the problem of subspace recovery.  In \secref{clustering}, we consider the problem of subspace clustering.  In both cases, we study a `canonical' RANSAC algorithm, deriving some theory and comparing it with other methods in numerical experiments.  We briefly discuss our results in \secref{discussion}.

\begin{rem}[linear vs affine]
Throughout, we consider the case where the subspaces are linear, although some applications may call for affine subspaces.  (This is for convenience.)  Because of this, we are able to identify a point $x \in \bbR^p$ with the corresponding vector (sometimes written $x - 0$).
\end{rem}

\section{Subspace recovery} \label{sec:recovery}

We consider the setting of \secref{recovery-intro} and use the notation defined there.  In particular, we work under \aspref{general}.  We consider the noiseless setting for simplicity.  

\subsection{RANSAC for subspace recovery}
We propose a simple RANSAC algorithm for robust subspace recovery.  
In the present setting, in particular under \aspref{general}, the underlying linear subspace $L$ (which we assumed is of dimension $d$) is determined by any $(d+1)$-tuple that comes from that subspace.  
The algorithm starts by randomly selecting a $(d+1)$-tuple and checking if this tuple forms a linear subspace of dimension $d$. 
If so, the subspace is recovered and the algorithm stops. 
Otherwise, the algorithm continues, repeatedly sampling a $(d+1)$-tuple at random until the subspace is discovered.  (Optionally, the algorithm can be made to stop when a maximum number of tuples has been sampled.)  In this formulation, detailed in \algref{ransac-recovery}, $d$ is known. 

\begin{algorithm}[h]
\caption{RANSAC (Subspace Recovery)}
\label{alg:ransac-recovery}
\SetKwInOut{Input}{Input}\SetKwInOut{Output}{Output}
\Input{data points $x_1, \dots, x_n \in \bbR^p$; dimension $d$}
\Output{a linear subspace of dimension $d$ containing at least $d+1$ points}
\Repeat{the tuple is linearly dependent}{randomly select a $(d+1)$-tuple of data points}
\Return{the subspace spanned by the tuple}
\end{algorithm}

By design, the procedure is exact.  (Again, we are in the noiseless setting.  In a noisy setting, the method can be shown to be essentially optimal.)  
However, researchers have shied away from a RANSAC approach because of its time complexity.  We formalize what is in the folklore in the following.

\begin{prp} \label{prp:ransac-recovery}
\algref{ransac-recovery} is exact and the number of iterations has the geometric distribution\footnote{ Here we consider the variant of the geometric distribution that is supported on the positive integers.} with success probability $\theta_1 := \binom{m}{d+1}/\binom{n}{d+1}$.  Thus the expected number of iterations is $1/\theta_1 = \binom{n}{d+1}/\binom{m}{d+1}$, which is of order $O(n/m)^{d+1}$ when $d$ is held fixed.
\end{prp}

Note that each iteration requires on the order of $O(p d^2)$ operations as it requires computing the rank of a $p$-by-$(d+1)$ matrix.  

\begin{proof}
The algorithm sample a $(d+1)$-tuple independently and uniformly at random until the tuple is linearly dependent.  Because of \aspref{general}, a $(d+1)$-tuple is linearly dependent if and only if all the points in the tuple are from $L$.  While there are $\binom{n}{d+1}$ $(d+1)$-tuples in total, only $\binom{m}{d+1}$ fit the bill, so that the probability of drawing a suitable tuple is $\theta_1 = \binom{m}{d+1}/\binom{n}{d+1}$.  Because the draws are independent, the total number of draws until the algorithm stops has the geometric distribution with success probability $\theta_1$.  

We know that the mean of this distribution is $1/\theta_1 = \binom{n}{d+1}/\binom{m}{d+1}$, and when $d$ is assumed fixed, while $n$ and $m$ are large, we have 
\[1/\theta_1 \sim \frac{n^{d+1}/(d+1)!}{m^{d+1}/(d+1)!} = (n/m)^{d+1}.\]
(The reader is invited to verify that this still holds true as long as $d = o(m^2)$.)
\end{proof}

In applications where the number of outliers is a non-negligible fraction of the sample (meaning that $n/m$ is not close to 1), the RANSAC's number of iterations  depends exponentially on the dimension of subspace.  This confirms the folklore, at least in such a setting.

\begin{rem} \label{rem:with-without}
For simplicity, we analyzed the variant of the algorithm where the tuples are drawn with replacement, so that the worst-case number of iterations  is infinite.  However, in practice one should draw the tuples without replacement (which is equally easy to do in the present setting), as recommended in \citep{schattschneider2012enhanced}.  For this variant, the worst-case time complexity is $\binom{n}{d+1}-\binom{n}{d+1} + 1$.  Moreover, \prpref{ransac-recovery} still applies if understood as an upper bound.  (The number of iterations  has a so-called negative hypergeometric distribution in this case.)
\end{rem}

\begin{rem} \label{rem:unknown-d}
If the dimension $d$ is unknown, a possible strategy is to start with $d=1$, run the algorithm for a maximum number of iterations, and if no pair of points is found to be aligned with the origin, move to $d=2$, and continue in that fashion, increasing the dimension.  If no satisfactory tuple is found, the algorithm would start again at $d=1$.  The algorithm will succeed eventually.
\end{rem}

\subsection{The algorithm of Hardt and Moitra for subspace recovery}
As we said above, researchers have avoid RANSAC procedures because of the running time, which as we saw can be prohibitive.  Recently, however, \cite{hardt2013algorithms} have proposed a RANSAC-type algorithm that strikes an interesting compromise between running time and precision.  

Their algorithm is designed for the case where the sample size is larger than the ambient dimension, namely $n > p$.  It can be described as follows.  It repeatedly draws a $p$-tuple at random until the tuple is found to be linearly dependent.  When such a tuple is found, the algorithm returns a set of linearly dependent points in the tuple.  See the description in \algref{hardt-moitra}.  A virtue of this procedure is that it does not require knowledge of the dimension $d$ of the underlying subspace. 

\begin{algorithm}[h]
\caption{Hardt-Moitra (Subspace Recovery)}
\label{alg:hardt-moitra}
\SetKwInOut{Input}{Input}\SetKwInOut{Output}{Output}
\Input{data points $x_1, \dots, x_n \in \bbR^p$}
\Output{a linear subspace}
\Repeat{the tuple is linearly dependent}{randomly select a $p$-tuple of data points}
\Return{the subspace spanned by any subset of linearly dependent points in the tuple}
\end{algorithm}

\begin{prp} \label{prp:hardt-moitra}
When $n > p$, \algref{hardt-moitra} is exact and its number of iterations  has the geometric distribution with success probability $\theta_2 := \sum_{k \ge d+1} \binom{m}{k}\binom{n-m}{p-k}/\binom{n}{p}$.  Thus the expected number of iterations  is $1/\theta_2$.
\end{prp}

Note that each iteration requires on the order of $O(p^3)$ operations as it requires computing the rank of a $p$-by-$p$ matrix.  

\begin{proof}
With \aspref{general} in place, a $p$-tuple is linearly dependent if and only if it contains at least $d+1$ points from the subspace $L$.  Thus the Repeat statement above stops exactly when it found a $p$-tuple that contains at least $d+1$ points.  Moreover, also because of \aspref{general}, the points within that tuple that are linear dependent must belong to $L$.  Therefore, the algorithm returns $L$, and is therefore exact.

We now turn to the number of iterations.  The number of iterations is obviously geometric and the success probability is the probability that a $p$-tuple drawn uniformly at random contains at least $d+1$ points from $L$.  $\theta_2$ is that probability.  Indeed, it is the probability that, when drawing $p$ balls without replacement from an urn with $m$ red balls out of $n$ total, the sample contains at least $d+1$ red balls.  In the present context, the balls are of course the points and the red balls are the points on the linear subspace.  
\end{proof}

\citet*{hardt2013algorithms} analyze their algorithm in a slightly different setting and with the goal of finding the maximum fraction of outliers that can be tolerated before the algorithm breaks down in the sense that it does not run in polynomial time.  In particular, they show that, if $m/n \ge d/p$, then their algorithm has a number of iterations  with the geometric distribution with success probability at least $1/(2p^2 n)$, so that the expected number of iterations is bounded by $2p^2 n$, which is obviously polynomial in $(p,n)$.
In fact, it can be better than that.  The following is a consequence of \prpref{hardt-moitra}.

\begin{cor}
If, in addition to $n > p$, it holds that $m/n \ge d/p$, with $d/p \le \tau$ and $p/n \le \tau$, for some fixed $\tau < 1$, then $\theta_2$ is bounded from below by a positive quantity that depends only on $\tau$.  Consequently, \algref{hardt-moitra} has expected number of iterations  of order $O(1)$.
\end{cor}

\begin{proof}
Let $U$ denote a random variable with the hypergeometric distribution with parameters $(p,m,n)$ described above.  Then $\theta_2 = \bbP(U \ge d+1)$, and it depends on $(d,p,m,n)$.  We show that $\theta_2$ is bounded from below irrespective of these parameters as long as the conditions are met.  Noting that $\theta_2$ is increasing in $d$ and $n$, and decreasing in $p$ and $m$, it suffices to consider how $\theta_2$ varies along a sequence where $n \to \infty$ and $(d,p,m)$ all varying with $n$ in such a way that $m/n \to \tau$ and $d/p \to \tau$, as this makes the expected number of iterations  largest.
Define
\begin{align*}
\mu &:= \E(U) = p (m/n); \\
\sigma^2 &= \Var(U) = p (m/n) (1-m/n) (n-p)/(n-1). 
\end{align*}
The condition $m/n \ge d/p$ implies that $\mu \ge d$, and along the sequence of parameters under consideration, $\sigma \to \infty$.  
Moreover, along such a sequence, $Z := (U - \mu)/\sigma$ is standard normal in the limit, so that
\begin{align*}
\theta_2 = \bbP(U \ge d+1) 
& = \bbP\big(Z \ge (d+1 - \mu)/\sigma\big) \\
& \ge \bbP\big(Z \ge 1/\sigma\big) \\
& \to \bbP\big(\cN(0,1) \ge 0) = 1/2,
\end{align*}
using Slutsky's theorem in the last line.
\end{proof}

\subsection{Numerical experiments}
We performed some small-scale numerical experiments comparing RANSAC (in the form of \algref{ransac-recovery}), the Hardt-Moitra (HM) procedure (\algref{hardt-moitra}), and the Geometric Median Subspace (GMS) of \citep{zhang2014novel}, which appears to be one of the best methods on the market.  (We used the code available on Teng Zhang's website.)

Each inlier is uniformly distributed on the intersection of the unit sphere with the underlying subspace.  
Each outlier is simply uniformly distributed on the unit sphere.  
The result of each algorithm is averaged over 1000 repeats.  Performance is measured by the (first principal) angle between the returned subspace and the true subspace.  (This is to be fair to GMS, as the other two algorithms are exact.)
The results are reported in \tabref{recovery}.  

We performed another set of experiments to corroborate the theory established in \prpref{ransac-recovery} for the complexity of RANSAC.  The results are shown in \figref{RANSAC recovery complexity}, where each setting has been repeated 1000 times.  As expected, as the dimensionality of the problem increases, RANSAC's complexity becomes quickly impractical.

\begin{table}
\centering
\begin{tabular}{@{}crrrcrrrcrrr@{}}\toprule
 \multicolumn{1}{c}{parameters} & \phantom{abc} & \multicolumn{3}{c}{average system time} & \phantom{abc}& \multicolumn{3}{c}{difference in angle} \\
 \cmidrule{3-5} \cmidrule{7-9} 
$(d,p,m,m_0)$ &&RANSAC & HM&GMS && RANSAC &HM& GMS \\ \midrule
$(8, 10, 100, 50)$ & &   .0051 &.0009&     .0947 && 0 &0&     .0341 \\
$(4,10,100,50)$ &  &    .0011 &.0006 &              .1912  && 0 &  0&   0 \\
$(8, 20,100,50)$ &  &    .0064&.0002&         .5008  && 0 &0&           .0184\\
$(6, 10, 100, 20)$ & & .0001 & .0006&        .1918 && 0 &0&     0 \\
$(9, 10, 100, 50)$ & &     .0076 & .0093&           .0117 && 0 &0&     .0411\\
$(18, 20, 100, 50)$ & &     .6123 &  .0013&       .0160&&0& 0 &       .2285\\
\bottomrule
\end{tabular}
\caption{Numerical experiments comparing RANSAC, HM, and GMS for the problem of subspace recovery.  As in the text, $d$ is the dimension of the subspace, $p$ is the ambient dimension, $m$ is the number of inliers, $m_0$ is the number of outliers (so that $n = m + m_0$ is the sample size).}
\label{tab:recovery}
\end{table}

\begin{figure}[!ht]
\center {\includegraphics[scale=0.5]{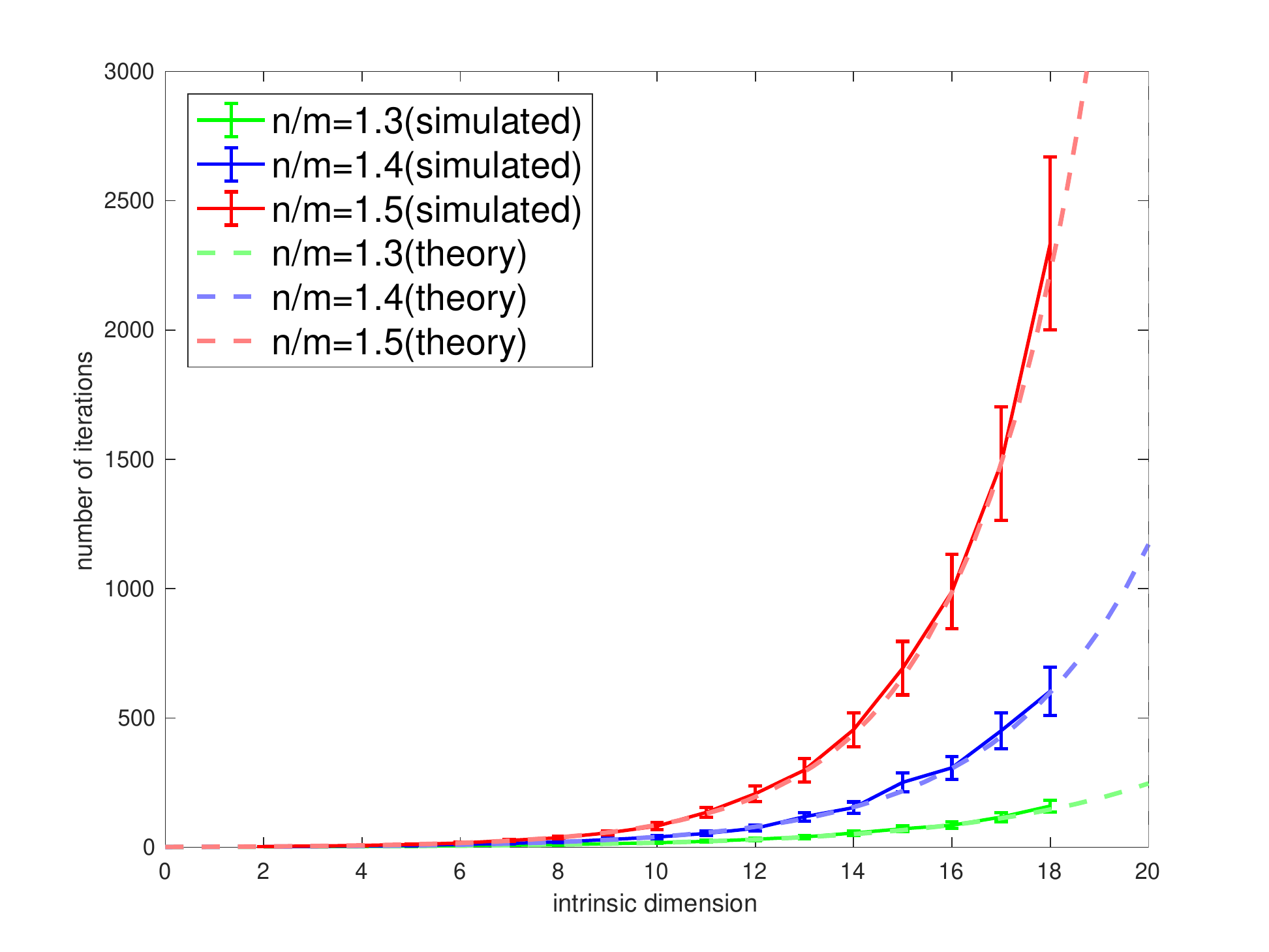}}
\caption{Average number of iterations for RANSAC (in the form of \algref{ransac-recovery}) as a function of the subspace dimension $d$ and the ratio of sample size $n$ to number of inliers $m$.  The dashed lines are the averages from our simulation while the lines are derived from theory (\prpref{ransac-recovery}).}\label{fig:RANSAC recovery complexity}
\end{figure}

\section{Subspace clustering} \label{sec:clustering}

We consider the setting of \secref{clustering-intro} and use the notation defined there.  In particular, we work under \aspref{general-clustering}.  We consider the noiseless setting for simplicity.  We also assume that all subspaces are of same dimension, denoted $d$ (so that $d_k = d$ for all $k$).

\subsection{RANSAC for subspace clustering}
We propose a simple RANSAC algorithm for subspace clustering.  As before, any of the linear subspaces is determined by any $(d+1)$-tuple that comes from that subspace.  
The algorithm starts by randomly selecting a $(d+1)$-tuple and checking if this tuple forms a linear subspace of dimension $d$.  If so, one of the subspaces is recovered and all the points on the subspace are extracted from the data.  Otherwise, the algorithm continues, repeatedly sampling a $(d+1)$-tuple at random until that condition is met.  The algorithm continues in this fashion until all the $K$ subspaces have been recovered.  
In this formulation, detailed in \algref{ransac-clustering}, both $d$ and $K$ are assumed known.

\begin{algorithm}[h]
\caption{RANSAC (Subspace Clustering)}
\label{alg:ransac-clustering}
\SetKwInOut{Input}{Input}\SetKwInOut{Output}{Output}
\Input{data points $x_1, \dots, x_n \in \bbR^p$; dimension $d$; number of subspace $K$}
\Output{$K$ linear subspaces of dimension $d$, each containing at least $d+1$ points}
\For{$k = 1, \dots, K$}{
\Repeat{the tuple is linearly dependent}{randomly select a $(d+1)$-tuple of data points}
\Return{the subspace spanned by the tuple}\\
remove the points on that subspace from the data
}
\end{algorithm}

Again, the procedure is exact by design, since we are in the noiseless setting.  Here too, researchers have not embraced RANSAC approaches because of their running time.  We confirm this folklore in the following, where we assume for simplicity that all subspaces have the same number of points $m$ (so that $m_k = m$ for all $k$).

\begin{prp} \label{prp:ransac-clustering}
\algref{ransac-clustering} is exact and the number of iterations is has the distribution of $I_1 + \dots + I_K$, where the $I$'s are independent and $I_j$ has the geometric distribution with success probability $(K-j+1) \binom{m}{d+1}/\binom{n - (j-1)m}{d+1}$.  This is stochastically bounded by the negative binomial with parameters $(K, \theta_1)$.  Thus the expected number of iterations is bounded by $K/\theta_1$, which is of order $O(n/m)^{d+1}$ when $d$ and $K$ are held fixed.
\end{prp}

The proof is very similar to that of \prpref{ransac-recovery} and is omitted.

\begin{rem}
When the dimensions of the subspaces are unknown, a strategy analogous to that described in \remref{unknown-d} is of course possible.  When the number of subspaces is unknown, a stopping rule can help decide whether there remains a subspace to be discovered.  Details are omitted as such an approach, although natural, could prove complicated.
\end{rem}

%\paragraph{The algorithm of Chen and Lerman}
%The algorithm of \cite{chen2009spectral} is of RANSAC type.  It was designed for the noisy setting and is therefore more complicated.\footnote{ \citet*{chen2009spectral} consider the case where the subspaces are affine, but it is easily adapted to the present case where we assume the subspaces to be linear instead.}
%However, in the noiseless setting, it essentially reduces to our simple RANSAC \algref{ransac-clustering}.

\subsection{Adapting the algorithm of Hardt and Moitra for subspace clustering}
\algref{ransac-clustering} consists in applying \algref{ransac-recovery} until a subspace is recovered, removing the points on that subspace, and then continuing, until all $K$ subspaces are recovered.  An algorithm for subspace clustering can be based on the algorithm of \cite{hardt2013algorithms} (\algref{hardt-moitra}) instead.  The resulting algorithm is suited for the case where $n - (m_1 + \dots + m_K) > p$.  Based on the fact that \algref{hardt-moitra} has expected number of iterations bounded by $2p^2 n$, the resulting algorithm for subspace clustering has expected number of iterations bounded by $2 K p^2 n$.  See \algref{hardt-moitra-clustering}, where we assume that the number of subspaces is known, but do not assume that the dimensions of the subspaces are known (and they do not need to be the same).

\begin{algorithm}[h]
\caption{Subspace Clustering based on the Hardt-Moitra Algorithm}
\label{alg:hardt-moitra-clustering}
\SetKwInOut{Input}{Input}\SetKwInOut{Output}{Output}
\Input{data points $x_1, \dots, x_n \in \bbR^p$; number of subspace $K$}
\Output{$K$ linear subspaces each with a number of points exceeding its dimension}
\For{$k = 1, \dots, K$}{
\Repeat{the tuple is linearly dependent}{randomly select a $p$-tuple of data points}
\Repeat{there are no more linearly dependent points in the tuple}{
find the smallest number of linearly dependent points in the tuple \\
\Return{the subspace spanned by these points} \\
remove the points on that subspace from the data}
}
\end{algorithm}

The reason why we extract the smallest number of linearly dependent points at each step is to avoid a situation where a $p$-tuple contains $d_j+1$ points from $L_j$ and $d_k+1$ points from $L_k$ (with $j \ne k$), in which case, assuming $d_j + d_k < p$, these points are linearly dependent but do not span one of the subspaces.  This particular step is, however, computationally challenging as it amounts to finding the sparsest solution to a $p$-by-$p$ linear system, a problem known to be challenging \citep[Eq 1]{tropp2010computational}.  One possibility is to replace this will finding the solution with minimum $\ell_1$ norm \citep[Eq 8]{tropp2010computational}.
The use of the $\ell_1$ constraint is central to the method proposed by \cite{elhamifar2009sparse}.

\subsection{The algorithm of Chen and Lerman}
The Spectral Curvature Clustering (SCC) algorithm of \cite{chen2009spectral} is in fact of RANSAC type.  The method was designed for the noisy setting and is therefore more sophisticated.\footnote{ \citet*{chen2009spectral} consider the case where the subspaces are affine, but we adapt their method to the case where they area linear.}
It is based on a function $\cA: (\bbR^p)^{d+1} \to [0,1]$ that quantifies how close a $(d+1)$-tuple is from spanning a subspace of dimension $d$ or less.  It is equal to 1 when this is the case and is strictly less than 1 when this is not the case.  
The algorithm draws a number, $c$, of $d$-tuples at random, where the $s$-th tuple is denoted $(x_{1,s}, \dots, x_{d, s})$, and computes the matrix $\bW = (W_{ij})$, where 
\[W_{ij} = \sum_{s = 1}^c \cA(x_i, x_{1,s}, \dots, x_{d, s}) \cA(x_j, x_{1,s}, \dots, x_{d, s}).\] 
It then applies a form of spectral graph partitioning algorithm to $\bW$ closely related to method of \cite{ng2002spectral}.
(The method assumes all subspaces are of same dimension $d$, and both $d$ and $K$ are assumed known.)

In the noiseless setting, one could take $\cA$ to return 1 if the tuple is linearly dependent and 0 otherwise.  In that case, $W_{ij}$ is simply the number of $d$-tuples among the $c$ that were drawn with whom both $x_i$ and $x_j$ are linearly dependent.  \cite{chen2009foundations} analyzes their method in a setting that reduces to this situation and show that the method is exact in this case.  

\subsection{Numerical experiments}
We performed some numerical experiments to compare various methods for subspace clustering, specifically, RANSAC, Sparse Subspace Clustering (SSC) \citep{elhamifar2009sparse}, Spectral Curvature Clustering (SCC) \citep{chen2009spectral}, and Thresholding-based Subspace Clustering (TSC) \citep{heckel2015theresholding}.  

Each inlier is uniformly distributed on the intersection of the unit sphere with its corresponding subspace.  
Each outlier is simply uniformly distributed on the unit sphere.  
The result of each algorithm is averaged over 500 repeats.   
Performance is measured by the Rand index.
The results are reported in \tabref{subspace}.

\newcommand{\ra}[1]{\renewcommand{\arraystretch}{#1}}

\begin{table}
\centering
\begin{tabular}{@{}cccccccccccc@{}}\toprule
 \multicolumn{1}{c}{parameters} & \multicolumn{4}{c}{average system time} &\phantom{a} & \multicolumn{4}{c}{rand index} \\
 \cmidrule{2-5} \cmidrule{7-10} 
$(d,p,K,m,m_0)$ &RANSAC & SSC  &SCC&TSC&& RANSAC & SSC  &SCC&TSC\\ \midrule
$(4,8,3,50,50)$ &    .0622 &     .3188&.3731&.1767  && 1 &         .8407&    .9997&    .7237\\
$(6,8,3,50,50)$ &.8810 &     .2813 &    .8257&.1815  && 1 &         .6490 &    .8322&    .5829\\
$(4,8,3,50,100)$ &.2108&     .3974 &.4563 &$.1757$& &  1 &       .7619 &.9689 &$.2849$ \\
$(4,8,5,50,50)$  &.3425 & .6163&    .9395&.3282&& 1 &.9206&   .9548&.7578\\
$(8,10,3,50,50)$  &   17.1827 & .3185&1.4141&    .2019  && 1 &.6148&.6904&.5688 \\

\bottomrule
\end{tabular}
\caption{Numerical experiments comparing RANSAC, SSC, SCC, and TSC for the problem of subspace clustering.  As in the text, $d$ is the dimension of the subspaces (assumed to be the same), $p$ is the ambient dimension, $K$ is the number of subspaces, $m$ is the number of inliers per subspace (assumed to be the same), $m_0$ is the number of outliers (so that $n = K m + m_0$ is the sample size).}
\label{tab:subspace}
\end{table}

\section{Discussion and conclusion} \label{sec:discussion}
In our small scale experiments, RANSAC is seen to be competitive with other methods, at least when the intrinsic dimensionality is not too large and when there are not too many outliers (or too many underlying subspaces) present in the data.  This was observed both in the context of subspace recovery and in the context of subspace clustering.

\bibliographystyle{chicago}
\bibliography{ref}

\begin{thebibliography}{}

\bibitem[\protect\citeauthoryear{Cand{\`e}s, Li, Ma, and Wright}{Cand{\`e}s
  et~al.}{2011}]{candes2011robust}
Cand{\`e}s, E.~J., X.~Li, Y.~Ma, and J.~Wright (2011).
\newblock Robust principal component analysis?
\newblock {\em Journal of the ACM (JACM)\/}~{\em 58\/}(3), 11.

\bibitem[\protect\citeauthoryear{Chen and Lerman}{Chen and
  Lerman}{2009a}]{chen2009foundations}
Chen, G. and G.~Lerman (2009a).
\newblock Foundations of a multi-way spectral clustering framework for hybrid
  linear modeling.
\newblock {\em Foundations of Computational Mathematics\/}~{\em 9\/}(5),
  517--558.

\bibitem[\protect\citeauthoryear{Chen and Lerman}{Chen and
  Lerman}{2009b}]{chen2009spectral}
Chen, G. and G.~Lerman (2009b).
\newblock Spectral curvature clustering (scc).
\newblock {\em International Journal of Computer Vision\/}~{\em 81\/}(3),
  317--330.

\bibitem[\protect\citeauthoryear{Elhamifar and Vidal}{Elhamifar and
  Vidal}{2009}]{elhamifar2009sparse}
Elhamifar, E. and R.~Vidal (2009).
\newblock Sparse subspace clustering.
\newblock In {\em Computer Vision and Pattern Recognition, 2009. CVPR 2009.
  IEEE Conference on}, pp.\  2790--2797. IEEE.

\bibitem[\protect\citeauthoryear{Fischler and Bolles}{Fischler and
  Bolles}{1981}]{fischler1981random}
Fischler, M.~A. and R.~C. Bolles (1981).
\newblock Random sample consensus: a paradigm for model fitting with
  applications to image analysis and automated cartography.
\newblock {\em Communications of the ACM\/}~{\em 24\/}(6), 381--395.

\bibitem[\protect\citeauthoryear{Hardt and Moitra}{Hardt and
  Moitra}{2013}]{hardt2013algorithms}
Hardt, M. and A.~Moitra (2013).
\newblock Algorithms and hardness for robust subspace recovery.
\newblock In {\em Conference on Learning Theory (COLT)}, Volume~30, pp.\
  354--375.

\bibitem[\protect\citeauthoryear{Huber and Ronchetti}{Huber and
  Ronchetti}{2009}]{huber2009robust}
Huber, P.~J. and E.~M. Ronchetti (2009).
\newblock {\em Robust Statistics (2nd edition)}.
\newblock Wiley.

\bibitem[\protect\citeauthoryear{Maronna}{Maronna}{1976}]{maronna1976robust}
Maronna, R.~A. (1976).
\newblock Robust {M}-estimators of multivariate location and scatter.
\newblock {\em The Annals of Statistics\/}~{\em 4\/}(1), 51--67.

\bibitem[\protect\citeauthoryear{Ng, Jordan, and Weiss}{Ng
  et~al.}{2002}]{ng2002spectral}
Ng, A.~Y., M.~I. Jordan, and Y.~Weiss (2002).
\newblock On spectral clustering: Analysis and an algorithm.
\newblock In {\em Advances in neural information processing systems}, pp.\
  849--856.

\bibitem[\protect\citeauthoryear{Reinhard~Heckel}{Reinhard~Heckel}{2015}]{heckel2015theresholding}
Reinhard~Heckel, Helmut~Bolcskei, D. o. I. . E. E. Z.~S. (2015).
\newblock Robust subspace clustering via thresholding.
\newblock {\em IEEE Transactions on Information Theory\/}~{\em 61}.

\bibitem[\protect\citeauthoryear{Schattschneider and Green}{Schattschneider and
  Green}{2012}]{schattschneider2012enhanced}
Schattschneider, R. and R.~Green (2012).
\newblock Enhanced ransac sampling based on non-repeating combinations.
\newblock In {\em Proceedings of the 27th Conference on Image and Vision
  Computing New Zealand}, pp.\  198--203. ACM.

\bibitem[\protect\citeauthoryear{Tropp and Wright}{Tropp and
  Wright}{2010}]{tropp2010computational}
Tropp, J.~A. and S.~J. Wright (2010).
\newblock Computational methods for sparse solution of linear inverse problems.
\newblock {\em Proceedings of the IEEE\/}~{\em 98\/}(6), 948--958.

\bibitem[\protect\citeauthoryear{Tyler}{Tyler}{1987}]{tyler1987distribution}
Tyler, D.~E. (1987).
\newblock A distribution-free m-estimator of multivariate scatter.
\newblock {\em The Annals of Statistics\/}, 234--251.

\bibitem[\protect\citeauthoryear{Vidal}{Vidal}{2011}]{vidal2011subspace}
Vidal, R. (2011).
\newblock Subspace clustering.
\newblock {\em IEEE Signal Processing Magazine\/}~{\em 28\/}(2), 52--68.

\bibitem[\protect\citeauthoryear{Vidal, Ma, and Sastry}{Vidal
  et~al.}{2005}]{vidal2005generalized}
Vidal, R., Y.~Ma, and S.~Sastry (2005).
\newblock Generalized principal component analysis (gpca).
\newblock {\em IEEE Transactions on Pattern Analysis and Machine
  Intelligence\/}~{\em 27\/}(12), 1945--1959.

\bibitem[\protect\citeauthoryear{Wright, Ganesh, Rao, Peng, and Ma}{Wright
  et~al.}{2009}]{wright2009robust}
Wright, J., A.~Ganesh, S.~Rao, Y.~Peng, and Y.~Ma (2009).
\newblock Robust principal component analysis: Exact recovery of corrupted
  low-rank matrices via convex optimization.
\newblock In {\em Advances in neural information processing systems}, pp.\
  2080--2088.

\bibitem[\protect\citeauthoryear{Zhang and Lerman}{Zhang and
  Lerman}{2014}]{zhang2014novel}
Zhang, T. and G.~Lerman (2014).
\newblock A novel {M}-estimator for robust {PCA}.
\newblock {\em The Journal of Machine Learning Research\/}~{\em 15\/}(1),
  749--808.

\end{thebibliography}

\end{document}